\documentclass[11pt]{amsart}

\usepackage{mathrsfs}
\usepackage{amssymb,amsmath,amsxtra,amscd,latexsym}
\usepackage{pstricks}
\usepackage{setspace,ifpdf}

\usepackage{pst-tree, pst-node}
\usepackage[mathscr]{eucal} 
\theoremstyle{plain}
\numberwithin{equation}{section}
\newtheorem{theorem}{Theorem}

\newtheorem{corollary}{Corollary}
\newtheorem{lemma}{Lemma}
\numberwithin{lemma}{section}
\numberwithin{corollary}{section}
\numberwithin{proposition}{section}

\theoremstyle{definition}
\newtheorem{definition}{Definition}
\numberwithin{definition}{section}

\theoremstyle{remark}
\newtheorem{example}{Example}
\numberwithin{example}{section}
\newtheorem{remark}{Remark}

\newcommand{\nc}{\newcommand}
\nc{\C}{\mathcal{C}}
\nc{\CC}{\widetilde{C}}
\nc{\wt}{\overline}
\nc{\mc}{\mathcal}
\nc{\on}{\operatorname}

\nc{\Iso}{\on{Iso}}

\nc{\m}{\mathbf{m}}
\nc{\N}{\on{N}}
\nc{\Nm}{\on{N}_{\m}}
\nc{\Cm}{\on{C}_{\m}}
\nc{\Nil}{\mathbf{Nil}}
\nc{\Perm}{\mathbf{Perm}}

\nc{\U}{\on{U}}
\nc{\Q}{\on{Q}}
\nc{\I}{\mathbf{I}}
\nc{\E}{\mathbf{E}}
\nc{\Ii}{\mathbb{I}}

\nc{\n}{\mathfrak{n}}
\nc{\g}{\mathfrak{g}}
\nc{\h}{\mathfrak{h}}
\renewcommand{\b}{\mathfrak{b}}

\nc{\HQ}{\mathbf{H}_{\Q}}
\nc{\HQe}{\widetilde{\mathbf{H}}_{\Q}}

\nc{\Fun}{\mathbb{F}_1}
\nc{\VFun}{\on{Vect}(\Fun)}
\nc{\RepQ}{\on{Rep}(\Q, \Fun)}
\nc{\RepQnot}{\on{Rep}(\Q_0, \Fun)_{nil}}
\nc{\kk}{\on{K}}
\nc{\RepQn}{\RepQ_{nil}}
\nc{\V}{\mathbb{V}}
\nc{\W}{\mathbb{W}}
\nc{\Z}{\mathbb{Z}}
\renewcommand{\dim}{\on{dim}}
\nc{\dimvv}{\underline{\dim}}
\renewcommand{\k}{\mathbf{k}}
\nc{\znz}{\mathbb{Z}/ n \mathbb{Z}}

\nc{\agln}{\hat{\mathfrak{gl}}_n}
\renewcommand{\a}{\mathfrak{a}}

\nc{\CQ}{\mathbf{C}_{\Q}}
\nc{\CQe}{\widetilde{\CQ}}

\begin{document}

\title{Representations of quivers over $\mathbb{F}_1$ and Hall algebras}
\author{Matt Szczesny}
\address{Department of Mathematics  
         Boston University, Boston MA, USA}
\email{szczesny@math.bu.edu}

\thanks{The author was supported by an NSA grant.}

%\date{June 2009}

\begin{abstract}

We define and study the category $\RepQ$ of representations of a quiver in $\VFun$ - the category of vector spaces "over $\Fun$". $\RepQ$ is an $\Fun$--linear category possessing kernels, co-kernels, and direct sums. Moreover, $\RepQ$ satisfies analogues of the Jordan-H\"older and Krull-Schmidt theorems. We are thus able to define the Hall algebra $\HQ$ of $\RepQ$, which behaves in some ways like the specialization at $q=1$ of the Hall algebra of $\on{Rep}(\Q, \mathbf{F}_q)$. We prove the existence of a Hopf algebra homomorphism of $ \rho': \U(\n_+) \rightarrow \HQ$, from the enveloping algebra of the nilpotent part $\n_+$ of the Kac-Moody algebra with Dynkin diagram $\overline{\Q}$ - the underlying unoriented graph of $\Q$.  We study $\rho'$ when $\Q$ is the Jordan quiver, a quiver of type $A$, the cyclic quiver, and a tree respectively. 

\end{abstract}

\maketitle 

\section{introduction}

Several recent papers \cite{CC1, CC2, CCM, D, D2, Du, H, LPL, Sou, TV} have been devoted to the development of algebraic geometry over $\Fun$ - the mystical field of one element. While this theory is still under development, there is agreement about what the definition of certain basic objects should be. In particular $\VFun$ - the category of vector spaces "over $\Fun$" is defined as follows. The objects of $\VFun$ are  pointed sets $(V, 0_V)$ and the morphisms $f : V \rightarrow W $ maps of pointed sets, having the additional property that $ f \vert_{V \backslash f^{-1}(O_W)}$ is an injection (such maps are also called \emph{partial bijections} in the literature). One hint that this is the correct definition comes from the phenomenon that linear algebra over finite fields $\mathbb{F}_q$ reduces in the limit $q \rightarrow 1$ to the combinatorics of finite sets. An example of this principle is the following. It is well-known that the Grassmannian  $Gr(k,n)$ of $k-planes$ in $\mathbb{F}^n_q$ has $\frac{[n]_q !}{[n-k]_q ! [k]_q !}$ points over $\mathbb{F}_q$, where
\[
[n]_q ! = [n]_q [n-1]_q \ldots [2]_q
\]
and 
\[
[n]_q = 1 + q + q^2 + \ldots + q^{n-1}
\]
In the limit $q \rightarrow 1$ this reduces to the binomial coefficient $\binom{n}{k}$, counting $k$--element subsets of an $n$--element set. 

The category $\VFun$ has almost all the good properties of the category of vector spaces over a field. All morphisms have kernels, cokernels, and transposes (duals),  it possesses a zero object, and has analogues of direct sum (pointed disjoint union) and tensor product (pointed Cartesian product). $\on{GL}(V)$ becomes the symmetric group.  The essential difference between $\VFun$ and an Abelian category is that $\on{Hom}(V,W)$ has no additive structure. It is simply a pointed set - i.e. itself an object in $\VFun$. 

This paper develops the representation theory of quivers in $\VFun$.  Given a quiver $\Q$ with vertex set $\I$ and edge set $\E$ we define a representation of $\Q$ in $\VFun$ to be an assignment of an object $V_i \in \VFun $ to each $i \in \I$, and a map $f_e : V_{e'} \rightarrow V _{e''}$ to each edge $e \in \E$ joining $e' \in \I$ to $e'' \in \I$. The category $\RepQ$ inherits most of the good properties of $\VFun$ - we have a zero object, kernels, cokernels, and direct sums. Moreover, the Jordan-H\"older and Krull-Schmidt theorems hold in $\RepQ$. $\RepQ$ therefore has enough structure to define its Hall algebra $\HQ$. As a vector space, $\HQ$ is defined to be the space of finitely supported functions on isomorphism classes of $\RepQ$ (denoted $\on{Iso}(\Q)$), i.e.
\[
\HQ := \{ f: \on{Iso}(\Q) : \rightarrow \mathbb{\mathbb{C}} \vert  \# (\on{supp}(f) ) < \infty \}
\]
It is equipped with an associative convolution product defined by 
\begin{equation} \label{prod1}
 (f \dot g)(M) = \sum_{ L \subset M} f(M/L) g(L) \; \; \; \;  M, L \in \on{Iso}(\Q)  
\end{equation}
This product encodes the structure of extensions in the category $\RepQ$. If we denote by $[M]$ the delta-function supported on the isomorphism class $M \in \on{Iso}(\Q)$, then \ref{prod1} can be written more explicitly as
\[
[M] \dot [N] = \sum_{R \in \on{Iso}(Q)} \frac{\mathbf{P}^{R}_{M,N}}{a_M a_N} [R]
\]
where $$\mathbf{P}^{R}_{M,N} = \# \{ \textrm{ short exact sequences } \mathbb{O} \rightarrow N \rightarrow R \rightarrow M \rightarrow \mathbb{O} \}$$
and $a_{M} = \# \on{Aut}(M)$.  
$\HQ$ carries a grading by the effective cone  $$K^+_0 (\RepQ) \simeq \mathbb{Z}_{\geq 0}^{|\I|} \subset K_0 (\RepQ),$$ by assigning to $[M]$ its class in $K_0 (\RepQ)$. 
$\HQ$ may also be given a coproduct $\Delta: \HQ \rightarrow \HQ \otimes \HQ$ defined by
\begin{equation*} \label{coprod}
\Delta(f)([M],[N]) = f([M \oplus N])  
\end{equation*}
The product and coproduct combine to give $\HQ$ the structure of a graded connected co-commutative bialgebra, and hence a graded Hopf algebra. By the Milnor-Moore theorem, $\HQ \simeq \U(\n_{\Q})$, where $\n_{\Q}$ is the Lie sub-algebra of primitive elements. Let $\overline{\Q}$ be the un-oriented graph underlying the quiver $\Q$. $\overline{\Q}$ may be viewed as the Dynkin diagram of a symmetric Kac-Moody algebra $\g(\overline{\Q})$ with triangular decomposition $$\g(\overline{\Q}) =\n_- \oplus \h \oplus \n_+ . $$ The main result of this paper is the following:

\begin{theorem} \label{main_thm}
Let $\Q$ be a quiver without self-loops (but with cycles allowed), and \mbox{ $ \g(\overline{\Q}) =\n_- \oplus \h \oplus  \n_+ $ } the symmetric Kac-Moody algebra with Dynkin diagram $\overline{Q}$. There exists a Hopf algebra homomorphism
\[
\rho : \U(\n_+) \rightarrow \HQ
\]
\end{theorem}
For a detailed statement see section \ref{rho}. In fact, $\rho$ may be extended to give a homomorphism from $\U(\b')$, where $\b'$ is the derived Borel subalgebra of $\g_{\overline{\Q}}$ to an extended version $\HQe$ of the Hall algebra. 
This result should be viewed as a degenerate version at $q=1$ of the following theorem, due to Ringel and Green \cite{R2, G}. 

\begin{theorem}
Let $\U_q (\g_{\overline{\Q}})$ be the quantum Kac-Moody algebra attached to $\overline{\Q}$, and $\U_q (\b')$ the (derived) quantum Borel subalgebra. Let $\widetilde{\mathbf{H}}_{\on{Rep}(\Q,q)}$ denote the extended Hall algebra of the category of representations of $\Q$ over the finite field $\mathbb{F}_q$. Then there exists a Hopf algebra homomorphism $\tau: \U_q (\b') \rightarrow \on{H}_{\on{Rep}(\Q,q)}$.  
\end{theorem}

The homomorphism $\tau$ may be formulated generically over the field $\mathbb{C}(q)$ (viewing $q$ as a variable), in which case it is injective. Moreover, it was shown by Ringel and Green that for $\Q$ of finite type, it is an isomorphism. $\rho$ is not a isomorphism in general, even for finite type quivers ( see example \ref{D_4}), but it is for type $A$, and an injection for $A^{1}_{n-1}$ (i.e. the cyclic quiver). We deduce in those cases the existence of a positive integral basis for $\U(\n_+)$ - a result proven earlier by Kostant \cite{Ko}. 

This paper is organized as follows. In section \ref{vspF1} we introduce the category $\VFun$, describe its basic properties, and deduce a Jordan normal form for endomorphisms. In section \ref{repQF1} we define the category $\RepQ$ of representations of $\Q$ in $\VFun$, describe its properties, and prove the Jordan-H\"older and Krull-Schmidt theorems. Section \ref{trees} gives a description of the indecomposable representations in the case when the graph underlying $\Q$ is a tree. 
Section \ref{Hall_alg} introduces the Hall algebra $\HQ$ of the category $\RepQ$, as well as an extended version $\HQe$. In section \ref{KM} we quickly recall some facts about Kac-Moody algebras. Theorem \ref{main_thm} is proved in section  \ref{rho}. The rest of the paper is devoted to examples. In section \ref{JQ} we consider the Jordan quiver and prove that its Hall algebra is isomorphic to the ring of symmetric functions. In \ref{typeA} we show that for quivers of type $\on{A}$, the homomorphism $\rho$ is an isomorphism. The cyclic quiver is treated in section \ref{cyclic}. Finally, in section \ref{further}, we indicate some natural further questions/directions.  

\bigskip

\noindent {\bf Acknowledgements:} The idea for this paper, as well as the spin-off \cite{Sz2} resulted from a conversation with Pavel Etingof. I am very grateful to him for several insightful suggestions and for pointing me in this direction. I would also like to thank Dirk Kreimer, Olivier Schiffmann, David Ben-Zvi, Valerio Toledano Laredo, and Hugh Thomas for valuable conversations. In particular, the review paper \cite{S} was a tremendous resource, and I have followed its exposition in several places. I am also grateful to the Universite Paris XIII for its hospitality while this paper was being finished. Finally, I would like to thank the anonymous referee for suggesting a number of improvements to this paper.

\section{Vector spaces over $\Fun$} \label{vspF1}

In this section we recall the category of vector spaces over $\Fun$ following \cite{KapS, H}. 
\begin{definition}
The category $\VFun$ of vector spaces over $\Fun$ is defined as follows. 
\begin{eqnarray*}
\on{Ob}(\VFun) & := & \{\textrm{ pointed finite sets } (V, 0_V ) \} \\
\on{Hom}(V,W) & := & \{ \textrm{ maps } f: V \rightarrow W \vert \; f(0_V) = 0_W \\ & & f \vert_{V \backslash f^{-1}(O_W)} \textrm{ is an injection }\}
\end{eqnarray*}
Composition of morphisms is defined as the composition of maps of sets, and so is associative. We refer to the unique morphism $f \in \on{Hom}(V,W)$ such that $f(V)= 0_W$ as the zero map. The \emph{dimension} of $V \in \VFun$ is $\dim(V) := |V| -1$. 
\end{definition}

$\on{Hom}(V,W)$ is a pointed set, with distinguished element the zero map. Thus \\ \mbox{ $\on{Hom}(V,W) \in \VFun$.} 
Given an inclusion of sets $X \subset Y$, we denote by $Y/X$ the pointed set obtained by identifying $X$ to a single point.  The category $\VFun$ has the following properties:

\begin{enumerate}
\item The unique object of $\VFun$ of dimension $0$ is an initial, terminal, and hence zero object, which we denote by $\emptyset$. 
\item Every morphism $f \in \on{Hom}(V,W)$ has a kernel $f^{-1}(0_W)$.
\item Every morphism $f \in \on{Hom}(V,W)$ has a cokernel $W/ f(V)$. 
\item $\VFun$ possesses a symmetric monoidal structure $V \oplus W$ defined as 
$$ V \oplus W := V \amalg W / \{ 0_V, 0_W \}. $$ There is an injection $\iota_V: V \hookrightarrow V \oplus W$ and a surjection $p_V : V \oplus W \rightarrow V$ such that $p_V \circ \iota_V = id_V$ (and same for $W$). 
\item $\VFun$ has another symmetric monoidal structure $V \otimes W$ defined as 
$$ V \otimes W := V \times W / \{ V \times 0_W \cup 0_V \times W \},$$ 
and $\oplus$, $\otimes$ satisfy the usual compatibilies
\[
(V \oplus W) \otimes Z \simeq V \otimes Z \oplus W \otimes Z \hspace{1cm} Z \otimes (V \oplus W) \simeq Z \otimes V \oplus Z \otimes W. 
\]
\item $\VFun$ has an involution (duality) sending $f \in \on{Hom}(V,W)$ to $f^t \in \on{Hom}(W,V)$, where $f^t(w) = f^{-1}(w)$ if $w \in f(V), w \neq 0_W$, and $f^t(w) = 0_V$ otherwise. 
\item For $V \in \VFun$, a \emph{subspace} of $V$ is a subset $U \subset V$ containing $0_V$. $(U,0_V)$ is itself naturally an object in $\VFun$. The intersection of two subspaces $U_1, U_2$ of $V$ is clearly a subspace of $V$, denoted $U_1 \cap U_2$. 
\item If $V, W \in \VFun$, and $U \subset V \oplus W$ is a subspace, then $$U = (U \cap V) \oplus (U \cap W).$$
\end{enumerate}
The first three properties above imply in particular that the notions of complex and exact sequence make sense in $\VFun$. We denote by $\k$ the one-dimensional vector space in $\VFun$, i.e. $\k$ is the unique object up to isomorphism such that $\dim(\k)=1$.  

\medskip
\begin{definition}
Given $V \in \VFun$ and a ring $R$, we define $V \otimes_{\Fun} R$ to be the free right $R$--module on the set $V \backslash 0_V$.
\end{definition}
The functor 
\begin{equation} \label{base_change}
()\otimes_{\Fun}R : \VFun \rightarrow \mathbf{R-Mod}
\end{equation}
 is exact ( i.e. takes short exact sequences in $\VFun$ to short exact sequences of $R$--modules ), and faithful. It is generally not full. 

\section{A normal form for endomorphisms} \label{norm_form}

We proceed to establish a version of the Jordan canonical form for endomorphisms over $\Fun$. For $V \in \VFun$, we denote by $\on{End}(V)$ the pointed set $\on{Hom}(V,V) \in \VFun$. Let $\on{GL}(V) \subset \on{End}(V)$ denote the subset of invertible endomorphisms. Elements of $GL(V)$ are simply permutations fixing $0_V$, and so $GL(V) \simeq \mc{S}_{\dim(V)}$, where $\mc{S}_n$ denotes the symmetric group on $n$ letters.  We say that $T \in \on{End}(V)$ is \emph{nilpotent} if there exists an $n \geq 0$ such that $T^n(x) = 0_V$ for all $x \in V$. 

\begin{definition} 
\begin{itemize}
\item Suppose that $V \in \VFun$, and $T \in \on{End}(V)$. We write \mbox {$T = \oplus_{j \in J} T_j$} if there exists a decomposition $V = \oplus_j V_j$ such that $T(V_j) \subset V_j$, and $T \vert_{V_j} = T_j$. In this case we say that $T$ \emph{is a direct sum of the }$T_j$. 
\item We say that $T$ \emph{is equivalent} to $T' \in \on{End}(W)$ if there exists and isomorphism $\sigma : V \rightarrow W$ such that $T = \sigma^{-1} T' \sigma$. 
\end{itemize}
\end{definition}

Let $\m \in \VFun$ denote the pointed set with elements $\{0_{\mathbf{m}}, 1, 2, \ldots, m \}$. Let \mbox{$\Nm \in \on{End}(\m)$} be the endomorphism defined by $$\Nm(k) = k-1, \; \; \;  k=1 \ldots, m, $$ and $\Cm \in \on{End}(\m)$ the endomorphism defined by $$\Cm(k) = k-1, \; \;  k= 2, \ldots, m, \; \; \Cm(1) = m .$$ $\Nm$ is the analogue of a $m \times m$ nilpotent Jordan block, whereas $\Cm$ is a cyclic permutation of $\m \backslash 0_{\mathbf{m}}$. 

\begin{lemma}[Jordan form] \label{Jordan}
Every $T \in \on{End}(V)$ can be written as a direct sum $T = \oplus_j T_j$ where $T_j \in \on{End}(V_j)$ is equivalent to either $\on{N}_{\mathbf{\dim(V_j)}}$ or $\on{C}_{\mathbf{\dim(V_j)}}$. This decomposition is unique up to the ordering of factors. 
\end{lemma}
  
\begin{proof}
Let $$\Nil \subset V := \{ x \in V | T^{n} (x) =  0_V \textrm{ for } n >> 0 \},$$ and let $$\Perm := V \backslash \Nil.$$ $T$ leaves both invariant, and so we have $ V = \Nil \oplus \Perm$. If $x \in V, x \neq 0_V$, then $x$ has at most one pre-image under $T$, which if nonempty, we denote $T^{-1}(x)$. We define $T^{-(j+1)}(x) := T^{-1}(T^{-j}(x))$, again allowing for the empty set. Let
\[
\mc{O}_x := \{ \ldots, T^{-2}(x), T^{-1}(x), x, T(x), T^{2}(x), \ldots \}.
\] 
$\mc{O}_x$ is invariant under $T$, and is a direct summand of $\Nil$ it it contains $0_V$, and $\Perm$ otherwise. In the first case, $T \vert{\mc{O}_x}$ is equivalent to $\on{N}_{|\mc{O}_x|-1}$, and in the latter to $\on{C}_{|\mc{O}_x |}$. 
Continuing in this fashion, we obtain the desired decomposition. It is clear that it is unique up to re-ordering. 
\end{proof}

\section{Representations of quivers in $\VFun$} \label{repQF1}

Recall that a \emph{quiver} is a directed graph, possibly with self-loops and/or multiple edges between vertices. We denote the set of vertices of the quiver by $\I$, the set of edges by $\E$, and let $r = \# \I$. For an edge $h \in \E$, we denote by $h', h''$ the source and target of $h$ respectively. 

\begin{definition}
A representation of a quiver $\Q$ in $\VFun$ is the following collection of data: 
\begin{itemize}
\item An assignment of an object $V_i \in \VFun$ to each vertex $i \in \I$.
\item For each edge $h \in \E$, a choice of $f_h \in \on{Hom} ( V_{h'},  V_{h''})$. 
\end{itemize}
We denote by $\mathbb{V}$ the data $(V_i, f_h), i \in \I, h \in \E$. 
\end{definition} 

The structure of representation of $\Q$ in $\VFun$ is equivalent to the structure of a module over the \emph{path monoid} of $\Q$. 

\begin{definition}
The category $\RepQ$ is defined as follows. 
\begin{eqnarray*}
\on{Ob}(\RepQ) & := & \{ \textrm{ representations } \V  \textrm{ of } \Q \textrm { in } \VFun \} \\
\on{Hom}(\V,\W) & := & \{ \Phi = (\phi_i)_{i \in \I}, \phi_i \in \on{Hom}(V_i, W_i),  g_h \circ \phi_{h'} = \phi_{h''} \circ f_h \}
\end{eqnarray*}
where $\V=(V_i, f_h)$ and $\W=(W_i, g_h)$. 
\end{definition}

\begin{definition} \begin{itemize}
\item The \emph{dimension} of a $\V \in \RepQ$ is defined to be $$\dim(\V) = \sum_{i \in \I} \dim(V_i).$$
\item The \emph{dimension vector} of $\V$ is the $|\I|$--tuple $$\dimvv(\V) = (\dim(V_i))_{ i \in \I}. $$
\end{itemize} 
\end{definition}

The category $\RepQ$ has the following properties:

\begin{enumerate}
\item The representation $\mathbb{O}$ of $\Q$ which assigns $0 \in \VFun$ to each vertex $i \in \I$ is a zero object.  
\item Every $\Phi \in \on{Hom}(\V,\W)$ has a kernel defined by $ker(\Phi)_i := ker (\phi_i)$. 
\item Every $\Phi \in \on{Hom}(\V,\W)$ has a co-kernel defined by $coker(\Phi)_i := coker(\phi_i)$. 
\item $\RepQ$ has a symmetric monoidal structure $\V \oplus \W$ defined by 
\[
\V \oplus \W := (V_i \oplus W_i, f_h \oplus g_h), \; i \in \I, h \in \E. 
\]
\item If $\mathbb{U} \subset \V \oplus \W$ is a sub-representation, then 
\begin{equation} \label{dir_sum_property}
\mathbb{U}= (\mathbb{U} \cap \V) \oplus (\mathbb{U} \cap \W).
\end{equation} 
\item If $\mathbb{U}$ and $\V$ are sub-representations of $\W \in \RepQ$, then so is $\mathbb{U} \cap \V$, defined by $(\mathbb{U} \cap \V)_i = U_i \cap V_i, \; \; i \in \I$, with edge maps restrictions of those of $\W$. 
\item $\on{Hom}(\V,\W)$ and $\on{Ext}^{n}(\V,\W)$ are finite sets, where the latter is defined as in the Yoneda approach as equivalence classes of $n+2$--step exact sequences. 
\end{enumerate}

In particular, the notions of a complex and exact sequence make sense in the category $\RepQ$. 
\medskip
Given $\V \in \RepQ$ and a field $\kk$, we may apply the base-change functor $() \otimes_{\Fun} \kk$ of section \ref{base_change} to each $V_i$ and the edge maps to obtain a functor
\begin{equation} \label{base_change2}
() \otimes_{\Fun} \kk : \RepQ \rightarrow \on{Rep}(\Q, \kk)
\end{equation}
This functor is exact and faithful, but not typically full. 

\medskip

In dealing with quivers which have cycles, it is sometimes useful to consider the full subcategory of $\RepQn$ of $\RepQ$ consisting of \emph{nilpotent} representations. 

\begin{definition}
A representation $\V \in \RepQ$ is \emph{nilpotent} if there exists $N \geq 0$ such that for any $n \geq N$, $f_{h_n} f_{h_{n-1}} \ldots f_{h_1} = 0$ for any path $h_1 h_2 \dots h_n$ in $Q$ (here we adopt the convention that paths are listed left-to-right in the order of traversal). 
\end{definition}

\begin{remark}
If $\Q$ does not have any cycles, then $\RepQn = \RepQ$. This can be seen by taking $N$ in the definition of nilpotent representation greater than the length of the longest path in $\Q$. 
\end{remark}

\medskip

We proceed to establish two basic structural results for the category $\RepQ$, namely the Jordan-H\"older and Krull-Schmidt theorems. 

\subsection{The Jordan-H\"older theorem} \label{JH}

\begin{definition}
$\V \in \RepQ$ is \emph{simple} if $\V$ does not contain a proper sub-representation $\W \subset \V$. 
\end{definition}

Let $S_i \in \RepQ$ denote the representation defined by the properties:
\begin{itemize}
\item $(S_i)_j = 0, \; j \neq i$, $(S_i)_i = \k$. 
\item All maps $f_h$ are $0$. 
\end{itemize}

\begin{lemma}
The only simple objects in $\RepQn$ are $S_i, \; i \in \I$. 
\end{lemma}

\begin{proof}
First, it is clear that the $S_i$ are simple. Let $\V = (V_i, f_h)$ be simple, and let $i_0$ be any vertex such that $V_{i_0} \neq 0$. Let $x \in V_{i_0}$ be a non-zero element. Since $\V \in \RepQn$ there exists a path $h_1 \ldots h_n$ such that $f_{h_n} \ldots f_{h_1}. x \neq 0$ but $f_{h_{n+1}} \ldots f_{h_1}.x =0$ for every edge $h_{n+1}$ leaving the terminal vertex $i_n$. We thus obtain an inclusion $S_{i_n} \hookrightarrow \V$, but since $\V$ is simple, this is an isomorphism.  
\end{proof}

It follows easily by induction on $\dim(\V)$ that every $\V \in \RepQ$ possesses a filtration 
\[
0 = F_0 \V \subset F_1 \V \subset \ldots \subset F_n V = \V
\]
 such that $F_i \V / F_{i-1} \V$ is simple. The following theorem shows that the simples occurring in the composition series are unique up to permutation. The proof given in \cite{E} for finite-dimensional modules over algebras goes through verbatim for the category $\RepQ$, and we reproduce the proof for completeness. 
 
 \begin{theorem}[Jordan-H\"older] \label{JH_thm}
 Let $0 = F_0 \V \subset F_1 \V \subset \ldots \subset F_n V = \V$ and $0 = F'_0 \V \subset F'_1 \V \subset \ldots \subset F'_{m} \V = \V$ be two filtrations of $\V$ such that the representations $ \mathbb{A}^j = F_j \V / F_{j-1} \V$ and $\mathbb{B}^j = F'_j \V /  F'_{j-1} \V$ are simple for all $j$. Then $m=n$, and there is a permutation $\sigma \textrm{ of } 1, \ldots, n$ such that $\mathbb{A}^j \simeq \mathbb{B}^{\sigma(j)}$ for $j = 1, \ldots, n. $
 \end{theorem}
 
 \begin{proof}
 The proof is by induction on $\dim(\V)$. If $\dim(\V) = 1$, then $\V$ is simple, and the theorem holds trivially. Suppose the theorem holds for all representations in $\RepQ$ of dimension at most $r$, and $\dim(\V) = r+1$. If $\mathbb{A}^1 = \mathbb{B}^1$ as subsets, then we are done since the theorem holds for $\V / \mathbb{A}^1 = \V / \mathbb{B}^1$. Otherwise, $\mathbb{A}^1 \neq \mathbb{B}^1$, which implies that $\mathbb{A}^1 \cap \mathbb{B}^1 = 0$, since they are simple. We therefore obtain an injection $f: \mathbb{A}^1 \oplus \mathbb{B}^1 \rightarrow \V$. Let $\mathbb{U} = \V / Im(f)$, and let $0=\bar{F}_0 \mathbb{U} \subset \ldots \subset \bar{F}_p \mathbb{U} = \mathbb{U}$ be a filtration of $\mathbb{U}$ with simple quotients $\Z^i= \bar{F}_i \mathbb{U} / \bar{F}_{i-1} \mathbb{U}$. Then 
 \begin{itemize}
 \item $\V/ \mathbb{A}^1$ has a filtration with successive quotients $\mathbb{B}^1, \Z^1, \ldots, \Z^p$, and another filtration with successive quotients $\mathbb{A}^2, \cdots, \mathbb{A}^n$. 
 \item $\V / \mathbb{B}^1$ has a filtration with successive quotients $\mathbb{A}^1, \Z^1, \ldots , \Z^p$, and another with successive quotients $\mathbb{B}^2, \ldots, \mathbb{B}^m$. 
 \end{itemize}  
 By the induction hypothesis, the collection (with multiplicities) $\{ \mathbb{A}^1, \mathbb{B}^1, \Z^1, \ldots, \Z^p \}$ coincides on one hand with $\{ \mathbb{A}^1, \ldots, \mathbb{A}^n \}$ and on the other with $\{  \mathbb{B}^1, \ldots, \mathbb{B}^m \}$, which implies the theorem. 
 \end{proof}
 
\begin{corollary} \label{K0}
$K_0 (\RepQn) = \mathbb{Z}^{|\I|}$, spanned by the classes of the $S_i, \; i \in \I$. 
\end{corollary}
 
\subsection{The Krull-Schmidt theorem} \label{KS}
 
 We say that $\V \in \RepQ$ is \emph{indecomposable} if it cannot be written as $\V=\mathbb{U} \oplus \W$, with $\mathbb{U}, \W \neq 0$. We say that $\Phi = (\phi_i)_{i \in I} \in \on{End}(\V)$ is \emph{nilpotent} if $\phi_i \in \on{End}(V_i)$ is nilpotent for every $i$. 
 
\begin{theorem}[Krull-Schmidt] \label{KS_thm}
Every $\V \in \RepQ$ can be uniquely (up to reordering) decomposed into a direct sum of indecomposable representations.  
\end{theorem}
 
 The proof proceeds in the same manner as in the case of finite-dimensional representations of algebras over a field. We begin with the following
 
 \begin{lemma} 
 If $\V \in \RepQ$ is indecomposable, then every $\Phi \in \on{End}(\V)$ is either nilpotent or an isomorphism. 
 \end{lemma}
 
 \begin{proof}
 Let $\Phi = (\phi_i)_{i \in I} \in \on{End}(\V)$. By lemma \ref{Jordan}, $\phi_i = \phi^{\Nil}_i \oplus \phi^{\Perm}_i$, where $\phi^{\Nil} \in \on{End}(V^{\Nil}_i)$ and $\phi^{\Perm}_i \in \on{End}(V^{\Perm}_i)$ are the nilpotent and invertible parts of $\phi_i$ respectively, and $V_i = V^{\Nil}_i \oplus V^{\Perm}_i$. Define $\V^{\Nil}$ by $(V^{\Nil})_i := V^{Nil}_i $ and $\V^{\Perm}$ by $(\V^{\Perm})_i = V^{\Perm}_i$, then one checks that $\V^{\Nil}$ and $\V^{\Perm}$ are subrepresentations, and so \mbox{$\V = \V^{\Nil} \oplus \V^{\Perm}$}. Since $\V$ is indecomposable, one of the factors must be $\mathbb{O}$. 
 \end{proof}
 
 \begin{proof}[Proof of theorem \ref{KS_thm}]
 
 One shows easily by induction on $\dim{\V}$ that $\V$ has a decomposition into indecomposable representations. Suppose that 
 \[
 \V = \V_1 \oplus \V_2 \oplus \ldots \oplus \V_k = \V'_1 \oplus \V'_2 \oplus \ldots \oplus \V'_l
 \]
 are two such decompositions. Denote by $\iota_s : \V_s \hookrightarrow \V$, $p_s : \V \rightarrow \V_s$ (resp. $\iota'_s, p'_s$) the inclusion and projection from/to $\V_s$ (resp. $\V'_s$). Let $\theta_r  := p_1 \iota'_r p'_r\iota_1 : \V_1 \rightarrow \V_1$. Suppose that at least two of the $\theta_r$ are non-zero, say $\theta_1 \textrm{ and } \theta_2$. Then neither is an isomorphism, and so by the previous lemma, they have to be nilpotent. However, since they are non-zero, there exist non-zero $x_1, x_2 \in \V_1$, $x_1 \neq x_2$ such that $\theta_i (x_i) = x_i$, contradicting the nilpotence. Thus, all but one, say $\theta_1$ are $0$, and $\theta_1: \V_1 \rightarrow \V_1$ is an isomorphism. This implies that $p'_1 \circ \iota_1 : \V_1 \rightarrow \V'_1$ is an isomorphism. 
 
 Let $\mathbb{S} = \oplus_{r > 2} \V_r$, and $\mathbb{S}' = \oplus_{s>2} \V'_s$, and let $\rho: \mathbb{S} \rightarrow \mathbb{S}'$ be the composition
 \[
 \rho: \mathbb{S} \hookrightarrow \V \rightarrow \mathbb{S}'
 \]
$\dim{\V}_1 = \dim{\V'}_1$ implies that  $\dim(\mathbb{S}) = \dim(\mathbb{S}')$, and  since $Ker(\rho) = 0$, $\rho$ is an isomorphism. We may now apply the above argument to the summand $\mathbb{S}$, proving the theorem. 
 
 \end{proof}
 
 \section{Quivers of tree type and their indecomposable representations} \label{trees}
 
 In this section, we show that for a quiver $\Q$ whose underlying graph is a tree (i.e. contains no cycles or multiple edges), the only indecomposable representations have dimension vectors containing only $1$'s and $0$'s. Such representations correspond to connected subgraphs of $\Q$. 
 
 \begin{definition}
 We say that a quiver $\Q$ is a \emph{tree} if the underlying graph is a tree (i.e. contains no cycles or multiple edges). 
 \end{definition} 
 
 \medskip
 
 \begin{theorem} \label{ind_tree}
 Let $\Q$ be a tree, and $\V$ an indecomposable representation of $\Q$. Then the entries of the dimension vector $\dimvv(\V)$ are all $0$'s and $1$'s. Indecomposable representations of $\Q$ correspond to connected sub-graphs of $Q$. 
 \end{theorem}
 
 \begin{proof}
 We introduce an equivalence relation $\sim$ on the set underlying $\V$ (which we will also denote by $\V$) as the equivalence relation generated by the relation $\smile$, defined as follows.  For any edge $h \in \E$ and $v_{h'} \in V_{h'}$, we declare $v_{h'} \smile f_h(v_{h'}) \in V_{h''}$.  Denote by $[v]$ the equivalence class of $v \in \V$. It is clear that for every $v \in \V$, $[v]$ is a sub-representation of $\V$, and that $$\V = [v] \oplus \left( \V /[v] \right). $$ It follows that if $\V$ is indecomposable, then $\V=[v]$ for any non-zero $v \in \V$. This in turn implies that the dimension vector consists of $0$'s and $1$'s, since if $v, v'$ are different elements of $V_i$, and $v \sim v'$, there would have to be an (un-oriented) path from $i$ to itself - i.e. a cycle, contradicting that $\Q$ is a tree. Finally, suppose that there exists an edge $h \in E$ such that $dim(V_{h'}) = dim(V_{h''}) = 1$, and $f_h = 0$. Then $\V$ can clearly be decomposed as a direct sum of two representations, supported on either side of $h$. It follows that for each edge $h \in E$ connecting one-dimensional vector spaces, $f_h$ is an isomorphism. Thus, $\V$ corresponds to a connected sub-graph of $Q$. 
 
 \end{proof}
 
 \section{The Hall algebra of $\RepQn$} \label{Hall_alg}

We proceed to define the Hall algebra of the category $\RepQ$ along the lines of \cite{S}.  Let $\on{Iso}(Q)$ denote the set of isomorphism classes of representations in $\RepQ$, and let us choose a representative for each. We will henceforth abuse notation and identify isomorphism classes with their representatives. Let 
\[
\mathbf{Z}_2 (Q) := \{ (X,R) | R \in \on{Iso}(\Q), X \subset R \}
\]
i.e. elements of $\mathbf{Z}_2 (\Q)$ are pairs of an isomorphism class and a subobject. Let  
\begin{equation}
\HQ := \{ f: \on{Iso}(\Q) : \rightarrow \mathbb{C} \vert  \# (\on{supp}(f) ) < \infty \}
\end{equation}
For $M \in \on{Iso}(\Q)$, we denote by $[M] \in \HQ$ the delta-function supported on the isomorphism class of $M$. $\HQ$ is thus the $\mathbb{C}$--vector space spanned by the symbols $[M]$. $\HQ$ can be given a product by:
\begin{equation} \label{prod}
 (f \cdot g)(M) = \sum_{ (L, M) \in \mathbf{Z}_2 (Q)} f(M/L) g(L) \; \; \;  M, L \in \on{Iso}(\Q)  
\end{equation}
This product is well-defined, since every object in $\RepQ$ has finitely many sub-objects, and is easily seen to be associative (see e.g. \cite{S}).  In the basis of the delta functions, we may write more explicitly:
\[
[M] \dot [N] = \sum_{R \in \on{Iso}(\Q)} \frac{\mathbf{P}^{R}_{M,N}}{a_M a_N} [R]
\]
where $$\mathbf{P}^{R}_{M,N} = \# \{ \textrm{ short exact sequences } \mathbb{O} \rightarrow N \rightarrow R \rightarrow M \rightarrow \mathbb{O} \}$$
and $a_{M} = \# \on{Aut}(M)$.  $\HQ$ carries a grading by the effective cone  \mbox{$K^+_0 (\RepQ) \simeq \mathbb{Z}_{\geq 0}^{|\I|}$} in $K_0 (\RepQ)$, by assigning to $[M]$ its class in $K_0 (\RepQ)$. This grading is clearly  preserved by the product \ref{prod}, and we may write
\[
\HQ = \oplus_{\alpha \in K^+_0 (\RepQ)} \HQ[\alpha],
\]
in terms of the graded pieces. 
$\HQ$ may be equipped with a coproduct \mbox{$\Delta: \HQ \rightarrow \HQ \otimes \HQ$} defined by
\begin{equation} \label{coprod}
\Delta(f)([M],[N]) = f([M \oplus N])  
\end{equation}
 
 \begin{lemma}
 $\Delta$ is co-associative, i.e. $(\Delta \otimes \on{Id}) \Delta = ( \on{Id} \otimes \Delta) \Delta$
 \end{lemma}
 
 \begin{proof}
 For $M \in \on{Iso}(Q)$, let  $$D_n (M) := \{ (A_1, \cdots, A_n) \in \on{Iso}(\Q)^n | A_1 \oplus \ldots \oplus A_n \simeq M \}$$
 We have 
 \[
 \Delta([M]) = \sum_{ (A,B) \in D_2 (M)} [A] \otimes [B]
 \]
 and 
 \[
 (\Delta \otimes \on{Id}) \Delta = \sum_{(A,B,C) \in D_3 (M)} [A] \otimes [B] \otimes [C] = ( \on{Id} \otimes \Delta) \Delta.
 \]
 \end{proof}
Note that $\Delta$ is compatible with the $K_0(\RepQ)$--grading on $\HQ$. Finally, we check that the multiplication and co-multiplication are compatible. 
 
 \begin{lemma}
$ \Delta([M] \cdot [N]) = \Delta([M]) \Delta([N])$ \; \; for $[M], [N] \in \on{Iso}(\Q)$
 \end{lemma}
 
 \begin{proof}
 For the left-hand side we have
 \begin{align*}
 \Delta([M] \cdot [N]) &= \Delta \left( \sum_{ \substack{ (X,R) \in \mathbf{Z}_2 (Q) \\ X \simeq N, R/X \simeq M }} [R] \right)\\
                               &= \sum_{ \substack{ (X,R) \in \mathbf{Z}_2 (Q) \\ X \simeq N, R/X \simeq M }} \left( \sum_{(A,B) \in D_2(R)} [A] \otimes [B] \right)                               
\end{align*}
 The terms appearing in the final expression are therefore in bijection with the set of ordered tuples of elements in $\mathbf{Z}_2 (Q) \times \on{Iso}(Q) \times \on{Iso}(Q)$
 \[
 \Pi := \{ [(X,R),A,B] | X \simeq N, R/X \simeq M, A \oplus B \simeq R \}
 \]
The right hand side yields
\begin{align*}
\Delta([M]) \Delta([N]) &= \sum_{(A_M, B_M) \in D_2 (M) } \sum_{(A_N, B_N) \in D_2 (N)} A_M \cdot A_N  \otimes B_M \cdot B_N \\
                                   &= \sum_{(A_M, B_M) \in D_2 (M) } \sum_{(A_N, B_N) \in D_2 (N)} \sum _{  \substack {(V, F) \in \mathbf{Z}_2 (Q) \\ V \simeq A_N, F/V \simeq A_M } } \sum _{ \substack {(G, W) \in \mathbf{Z}_2 (Q) \\ G \simeq B_N, W/G \simeq B_M  }} V \otimes W
\end{align*}
 The terms appearing in the final expression are therefore in bijection with the set of ordered tuples of elements in $\mathbf{Z}_2 (Q) \times \mathbf{Z}_2 (Q)$
 \[
 \Omega := \{ [(V,F),(W,G)] | V \oplus W \simeq N, F/V \oplus G/W \simeq M \}
 \]
 By \ref{dir_sum_property}, if $X \subset A \oplus B$, then $X \simeq X \cap A \oplus X \cap B$. We can therefore define a map $j: \Pi \rightarrow \Omega$ given by $j( [(X,R),A,B] ) = [(X \cap A, A), (X \cap B, B)]$, which is easily seen to be a bijection. 

\end{proof}

$\HQ$ is thus a graded, connected, and co-commutative Hopf algebra. By the Milnor-Moore theorem, it is the universal enveloping algebra of the graded Lie algebra $\n_Q$ of its primitive elements, i.e. $\HQ \simeq \U(\n_Q)$. It follows from the definition of the coproduct \ref{coprod} that $[M] \in \HQ$ is primitive iff $M$ is indecomposable. We therefore obtain:

\begin{theorem} \label{Hall_isom}
$\HQ \simeq \U(\n_Q)$, where $\n_Q$ is the pro-nilpotent Lie algebra spanned by $[M]$, for $M \in \on{Iso}(\Q)$ indecomposable. 
\end{theorem}

It follows from this in particular that $\HQ$ is generated as an algebra by $\n_Q$. 

 \subsection{The extended Hall algebra $\HQe$} \label{HQe}
 
 We may extend $\HQ$ by adjoining to it a degree zero piece coming from $K_0 (\RepQ)$.
 We define $\h_{\Q} = \on{Hom}_{\mathbb{Z}} (K_0 (\RepQ), \mathbb{C})$, and so obtain canonical identifications \mbox{$\h^*_{\Q} \simeq K_0 (\RepQ) \otimes_{\mathbb{Z}} \mathbb{C}$}, and $\on{Sym}(\h_{\Q}) \simeq \mathbb{C}[K_0 (\RepQ) \otimes_{\mathbb{Z}} \mathbb{C}]$ (where for an affine variety $Y$,  $\mathbb{C}[Y]$ denotes the ring of regular functions on $Y$). Let now
 \begin{equation}
 \HQe := \on{Sym}(\h_{\Q}) \otimes_{\mathbb{C}} \HQ
 \end{equation}
 on which we impose the relations
 \begin{equation} \label{relation}
 [Z,f] := Z(\alpha) f \; \textrm{ for } Z \in \h_{\Q},  f \in \HQ[\alpha]
 \end{equation}
 With this twisted multiplication, $\HQe$ is easily seen to be isomorphic as an associative algebra to $\U(\b_{\Q})$, where as a vector space $$ \b_{\Q} := \h_{\Q} \oplus \n_{\Q},$$ with Lie bracket between $\h_{\Q}$ and $\n_{\Q}$ given by \ref{relation}. As $\U(\b_{\Q})$ is a Hopf algebra containing $\U(\n_{\Q}) \simeq \HQ$, we may extend the coproduct \ref{coprod} to $\HQe$ by requiring that for $Z \in \h_{\Q}$,  $$\Delta(Z) = Z \otimes 1 + 1 \otimes Z \in \HQe \otimes \HQe,$$
 and using the fact that $\Delta$ is an algebra homomorphism. 
 
 \begin{remark} \label{which_category}
 The Hall algebra construction can equally well be applied to the full subcategory $\RepQn \subset \RepQ$ (and more generally to any \emph{finitary} category, see \cite{R, S}). In the following sections, we will always restrict our attention to the Hall algebra of this smaller category.  Theorem \ref{Hall_isom} remains unchanged, with $\n_{\Q}$ now spanned by the isomorphism classes of the indecomposable nilpotent representations. Recall that $\RepQn$ differs from $\RepQ$ only for those quivers that have cycles. 
 \end{remark}
 
 \section{Recollections on Kac-Moody algebras} \label{KM}
 
 Given a quiver $\Q$, denote by $\overline{\Q}$ the underlying unoriented graph. We will assume that $\Q$ has no self-loops (but cycles are allowed).  We may attach to $\overline{\Q}$ a symmetric (i.e. simply laced) Kac-Moody algebra $\g=\g(\overline{\Q})$ as follows (see \cite{Kac} for details regarding statements in this section). Let $A$ be the \emph{generalized Cartan matrix} corresponding to $\overline{\Q}$; that is, the symmetric matrix $r \times r$ matrix whose diagonal entries are $2$, and where the off-diagonal entries 
 $$ a_{ij} = A_{ij} := - \# \{ \textrm{ edges joining } i \textrm{ with } j  \textrm{ in either direction } \}, \; \forall i,j \in I. $$
 A \emph{realization} of $A$ consists of a pair of complex vector spaces $\h, \h^*$ containing sets of vectors $$\Pi = \{ \alpha_1, \cdots, \alpha_r \} \subset \h^* \textrm{ and } \Pi^{\spcheck} = \{ h_1, \cdots, h_r \} \subset \h$$ satisfying the following three conditions:
 \begin{enumerate}
 \item $\Pi, \Pi^{\spcheck}$ are linearly independent
 \item $dim(\h) = 2r - rank(A)$
 \item $\alpha_i (h_j) = a_{ji}$
 \end{enumerate} 
 The Kac-Moody algebra $\g=\g(\overline{\Q})$ associated to $A$ is the Lie algebra generated by elements $\{e_i, f_i, h  | i \in \I, h \in \h \}$ subject to the relations
 \begin{align} \label{Serre_relations}
  & [h,h'] = 0 \\
 & [h,e_j] = \alpha_j (h) f_j \notag \\
 & [h, f_j] = - \alpha_j (h) f_j \notag \\
 & [e_i, f_j] = \delta_{ij} h_i \notag \\
 & ad^{1-a_{ij}}(e_i)(e_j) = 0 \notag \\
 & ad^{1-a_{ij}}(f_i)(f_j) = 0 \notag
  \end{align}
 We denote by $\n_+, (\textrm{ resp.  } \n_{-})$ the subalgebras generated by the $e_i, ( \textrm{ resp. } f_i)$. $\g$ has a triangular decomposition
\[
\g = \n_{-} \oplus \h \oplus \n_{+}
\] 
which induces at the level of enveloping algebras the decomposition (as vector spaces)
\[
\U(\g) = \U(\n_{-}) \otimes_{\mathbb{C}} \U(\h) \otimes_{\mathbb{C}} \U(\n_+)
\]
The Lie algebra $\n_{+}$ (resp. $\n_{-}$) can be described without reference to $\g$ as the Lie algebra generated by $\{ e_i \}, \; i \in \I$ (resp. $ \{ f_i \}, \; i \in \I $) subject to the second-to last (resp. last) of the relations \ref{Serre_relations}.  We will find it useful to consider below also the derived subalgebra $\g'=[\g,\g]$. It is known that $$\g' = \n_{-} \oplus \bigoplus^r_{i=1} \mathbb{C} h_i \oplus \n_{+}$$ and so is generated by $e_i, f_i, \textrm{ and } h_i, \; i \in \I$.  The subalgebra $\b = \h \oplus \n_{+} \subset \g$ is called the standard Borel subalgebra, and $\b' = \bigoplus^r_{i=1} \mathbb{C} h_i \oplus \n_{+}$ the derived Borel subalgebra.

%\begin{remark}
%The construction of the Kac-Moody algebra $\g=\g_Q$ does not depend on the orientation of $Q$ - only on the underlying un-oriented graph. 
%\end{remark}

 \section{The homomorphism $\rho: \U(\b') \rightarrow \HQe $} \label{rho}
 
In this section, we construct a Hopf algebra homomorphism $\rho : \U(\b') \rightarrow \HQe$ where $\b'$ is the derived Borel subalgebra of the symmetric Kac-Moody algebra $\g = \g(\overline{\Q})$ with Dynkin diagram the underlying unoriented graph of $\Q$, and $\HQe$ is the extended Hall algebra of the category $\RepQn$. Let $Z_i \in \h_{\Q}$ be element uniquely determined by the requirement that $Z_i (\alpha_j) = a_{ji}$. 

\begin{theorem} \label{rho_theorem}
The assignment $e_i \rightarrow [S_i], h_i \rightarrow Z_i$ extends to a Hopf algebra homomorphism 
\[
\rho: \U(\b') \rightarrow \HQe
\]
\end{theorem}

\begin{proof}
To verify that $\rho$ is an algebra homorphism, we must check that the images of $e_i, h_i \; i \in \I$ satisfy the relations \ref{Serre_relations}.  
Only the first, second, and fifth of the relations \ref{Serre_relations} are relevant.  The first two are immediate - it follows by construction that:
\begin{enumerate}
\item $[Z_i, Z_j] =0$ 
\item $[Z_i, [S_j]] = Z_i (\alpha_j) = a_{ji}$
\end{enumerate}
We have to verify the fifth relation, $ad^{1-a_{ij}}([S_i])([S_j]) = 0$, which may be written as 
\begin{equation} \label{main_relation}
\sum^{1-a_{ij}}_{l=0} (-1)^l \binom{1-a_{ij}}{l} [S_i]^l [S_j] [S_i]^{1-a_{ij} - l} = 0. 
\end{equation}
Let us introduce the divided powers $[S_i]^{(l)} = \frac{[S_i]^l}{l !} = [S^{\oplus l}_i]$. In terms of these, the relation \ref{main_relation} reads
\begin{equation} \label{main_relation2}
\sum^{1-a_{ij}}_{l=0} (-1)^l [S_i]^{(l)} [S_j] [S_i]^{(1-a_{ij} - l)} = 0. 
\end{equation}
For non-negative integers $l,n$ let us write
\[
[S_i]^{(l)} [S_j] [S_i]^{(n)}  = \sum_{[M] \in \on{Iso}(Q)} G^{M}_{l,n} [M]
\]
The non-negative integer $G^{M}_{l,n}$ is the number of three-step filtrations 
\begin{equation} \label{filtration}
 \emptyset \subset F^1 M \subset F^2 M \subset F^3 M = M 
 \end{equation}
such that $F^1 M \simeq S^{\oplus n}_i, \; F^2 M / F^1 M \simeq S_j, \; \textrm{ and } F^3 M / F^2 M \simeq S^{\oplus l}_I$.  
Let 
\[
U_M = \bigcap_{\substack{ h \in \E \\ h' = i, h''=j}} \on{Ker}(f_h)
\]
and
\[
V_M = \bigcup_{\substack{ h \in \E \\ h' = j, h''=i}} \on{Im}(f_h),
\]
and denote their dimensions by $u_M$ and $v_M$ respectively. 
In order for a filtration \ref{filtration} to exist, it is necessary and sufficient that $$ V_M \subset F^1 M \subset U_M, $$ and so there are 
$\binom{u_M - v_M}{ n - v_M}$ such filtrations.  It follows that for a fixed $M$,
\begin{align*}
\sum^{1-a_{ij}}_{l=0} (-1)^l G^M_{l,1-a_{ij} - l}  & = \sum^{1-a_{ij}}_{l=0} (-1)^l \binom{u_M - v_M}{1-a_{ij} - l - v_M} \\
&= \pm \sum^{u_M - v_M}_{k=0} (-1)^k \binom{u_M - v_M}{k} \\
& = 0
\end{align*}
We have shown that $\rho$ is a homomorphism of associative algebras. To see that it is a Hopf algebra homomorphism, it suffices to observe that $e_i, h_i, \; i \in \I$ (resp. $[S_i], Z_i$) are primitive, and since the former generate $\U(\b')$, the statement follows.
\end{proof}

\begin{remark}
Restricting $\rho$ to $\U(\n_+)$ we obtain a Hopf algebra homomorphism $$\rho': \U(\n_+) \rightarrow \HQ .$$ Since both algebras are enveloping algebras, we see that $\rho$ and $\rho'$ are determined by the corresponding Lie algebra homorphism $$ \rho_{Lie} : \b_+ \rightarrow \b_{\Q}.$$
\end{remark}

The homomorphism $\rho$ need not be injective, as can be seen from the following:
 
 \begin{example} \label{D_4}
 \setlength{\unitlength}{1pt}
 We consider the example of $D_4$, with the Dynkin quiver oriented as follows:
 
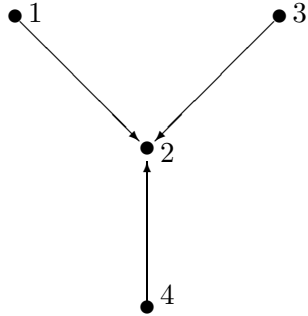
\begin{center}
 \begin{picture}(120,120)
 \put(50,50){\circle*{5}}
 \put(0,100){\circle*{5}}
 \put(100,100){\circle*{5}}
 \put(50,-10){\circle*{5}}
 \put(2,98){\vector(1,-1){45}}
 \put(98,98){\vector(-1,-1){45}}
 \put(50,-8){\vector(0,1){53}}
 \put(5,98){$1$}
 \put(105,98){$3$}
 \put(55,-9){$4$}
 \put(55,45){$2$}
 \end{picture}
 \end{center}
 
 \vspace{1cm}
 
 The root system is 
 \begin{eqnarray*}
 \{ \alpha_1, \alpha_2, \alpha_3, \alpha_4 \}  \cup \{ \alpha_1+\alpha_2, \alpha_2+\alpha_3, \alpha_2+\alpha_4 \} \\
  \cup \{ \alpha_1 + \alpha_2 + \alpha_4, \alpha_1 + \alpha_2 + \alpha_3, \alpha_2 + \alpha_3 + \alpha_4 \}  \\
 \cup \{ \alpha_1 + \alpha_2 + \alpha_3 + \alpha_4,   \alpha_1 + 2 \alpha_2 + \alpha_3 + \alpha_4 \}  \\
 \end{eqnarray*}
As before, we denote by $S_i, i=1 \ldots 4$ the simple one-dimensional representation supported at $i$, by $R_{ij}, R_{jik}, \textrm{ etc. }$ the unique indecomposable representation supported on vertices $i, j$ (resp. $i,j,k$ etc.) with one-dimensional spaces at each vertex and all maps between non-zero spaces isomorphisms. The list of indecomposable elements of $\RepQ$ is 
\[
\{S_1, S_2, S_3, S_4 \} \cup \{ R_{12}, R_{23}, R_{24} \} \cup \{ R_{123}, R_{124}, R_{234} \} \cup \{ R_{1234} \}
\]
Thus, the indecomposable corresponding to the maximal root $\alpha_1 + 2 \alpha_2 + \alpha_3 + \alpha_4$ is missing over $\mathbb{F}_1$. In $\on{Rep}(\Q, \mathbb{F}_q)$, this indecomposable corresponds to three distinct lines in $\mathbb{F}^{\oplus 2}_q$, or in other words, three points of $\mathbb{P}^1_{\mathbb{F}_q}$, but over $\mathbb{F}_1$, this cannot happen. 

We denote by $e_{i}, e_{ij}, e_{ijk}$ basis vectors for the one-dimensional root spaces $\alpha_i, \alpha_i+\alpha_j, \alpha_i + \alpha_j + \alpha_k $ etc. We may choose
\[
\begin{array}{l l l l}
e_1 = E_{12} - E_{65}  & e_{12} = E_{24} - E_{86} & e_{123} = E_{14} - E_{85} & e_{1234} = E_{17} - E_{35} \\
e_2  = E_{23} - E_{76} & e_{23} = E_{24} - E_{86} & e_{124} = E_{18} - E_{45} & e_{12^2 34} = E_{16} - E_{25} \\ 
e_3  = E_{34} - E_{87}  & e_{24} = E_{28} - E_{46}& e_{234} = E_{27} - E_{36} & \\
e_4  = E_{38} - E_{47}  & & & \\
\end{array} 
\]

Starting with $\rho(e_i) = S_i$ we obtain that $\rho(e_{ij}) = R_{ij}$, $\rho(e_{ijk}) = R_{ijk}$, $\rho(e_{1234}) = R_{1234}$, and $\rho(e_{12^234}) = 0$. Thus $Ker(\rho)$ is the two-sided ideal generated by $e_{12^234}$. 

 \end{example}

%\begin{remark}
%It is clear from this example that if $\Q$ has no cycles, then any indecomposable representation $\V \in \RepQ$ has dimension vector $\dimvv(\V)$ consisting only of $0$'s and $1$'s. if $\Q$ has cycles, we may obtain representations that "wrap around" the cycle (see the example of the cyclic quiver below) and have dimension vectors with larger entries. 
%\end{remark}

\begin{definition}
The Hopf sub-algebras $\CQ := \on{Im}(\rho') \subset \HQ$, $\CQe :=\on{Im}(\rho) \subset \HQe$ are called the \emph{composition} and \emph{extended composition} algebras respectively. 
\end{definition}

In general, $\CQ, \CQe$ are proper Hopf sub-algebras (see the example of the cyclic quiver below). In the case when $\Q$ is a tree, $\rho'$ and $\rho$ are surjective:

\begin{theorem}
$\rho$ and $\rho'$ are surjective whenever $\Q$ is a tree. 
\end{theorem}

\begin{proof}
It suffices to show that $\rho': \U(\n_+) \rightarrow \HQ$ is surjective, and for this, that every indecomposable representation of $\Q$ lies in $\CQ$. We know from theorem \ref{ind_tree} that indecomposable representations of $\Q$ correspond to connected subgraphs (i.e. sub-trees) of $\Q$. Let $\overline{\Q}' \subset \overline{\Q}$ be such a subgraph, and let $i \in \I$ be an extremal vertex of $\overline{\Q}'$ (a vertex lying at the end of branch of $\overline{\Q}'$). Removing $i$ from $\overline{\Q}'$ results in a connected sub-tree $\overline{\Q}'' \subset \overline{\Q}' \subset \overline{\Q}$. Let $\V_{\Q''}, \V_{\Q'}$ be the indecomposable representations of $\Q$ corresponding to the subgraphs $\overline{\Q}', \overline{\Q}''$. In the Hall algebra $\HQ$ we have the relation $$  \left[ [S_i], [\V_{\Q''}] \right] = \pm [\V_{\Q'}],  $$ where the sign depends on the orientation of the quiver. The theorem now follows by induction on the number of vertices of $\overline{\Q}'$. 
\end{proof}

\section{The Jordan Quiver} \label{JQ}

In this section, we consider the example of the Jordan quiver $\Q_0$, consisting of a single vertex $i$ and a single loop $h: i \rightarrow i$. 
For a treatment of the Hall algebra of $\on{Rep}(\Q_0, \mathbb{F}_q$) see \cite{S}. We will restrict ourselves to the category $\RepQnot$ of nilpotent representations of $Q_0$. An object of $\RepQnot$ consists of a pair $(V_i, f_h)$ where $ V_i  \in \VFun $ and $f_h : V_i \rightarrow V_i$ is nilpotent. By the decomposition theorem \ref{Jordan}, the indecomposable objects of $\RepQnot$ are exactly the $\N_m$ introduced in section \ref{norm_form}, with $\N_1$ the only simple object. A general representation is a direct sum 
\begin{equation} \label{obj}
\N^{\oplus n_1}_1 \oplus \N^{\oplus n_2}_2 \oplus \ldots \oplus \N^{\oplus n_r}_r,
\end{equation}
for non-negative integers $n_1, \ldots, n_r$, and so we may identify isomorphism class \ref{obj} with the partition
$$\lambda = (1^{n_1} 2^{ n_2} \cdots r^{ n_r}).$$ 
We will also use the notation $\lambda=(\lambda_1, \lambda_2, \ldots, 0, 0, \ldots)$ for partitions, which exhibits lambda as a non-increasing sequence of  non-negative integers, i.e. $\lambda_1 \geq \lambda_2 \geq \ldots$. 
By theorem \ref{Hall_isom} we have that $\mathbf{H}_{\Q_0} \simeq \U(\n_{\Q_0})$, where is $\n_{\Q_0}$ is the Lie algebra spanned by $[\N_i], \; i \geq 1$. We have in $\mathbf{H}_{\Q_0}$
\begin{equation} \label{symm_rel}
[\N_i][\N_j] = [\N_i \oplus \N_j] + [\N_{i+j}]
\end{equation}
from which it follows that $\n_{\Q_0}$ is abelian and $\mathbf{H}_{\Q_0}$ therefore commutative. Let $\mathbf{\Lambda}$ denote the ring of symmetric functions with complex coefficients. One of the standard bases for $\mathbf{\Lambda}$ consists of the monomial symmetric functions $m_{\lambda}$. These are defined as follows:
\[
m_{\lambda} = \sum_{\alpha} x^{\alpha}
\]
where the summation is over all distinct permutations $\alpha$ of the entries $(\lambda_1, \lambda_2, \ldots, 0 , 0 , 0 \ldots)$ of the partition $\lambda$, and we are using multi-index notation. It is well-known (see for instance \cite{St}) that for $i < j$ 
\[
m_{(i)} m_{(j)} = m_{(j, i)} + m_{(i+j)}
\]
and that $\mathbf{\Lambda}$ is freely generated by the $m_{(i)}, i \in \mathbb{N}$ (these are called the power sum symmetric functions). We thus obtain a surjective ring homomorphism
\[
\mathbf{\Lambda} \rightarrow \mathbf{H}_{\Q_{0}}
\]
by sending $m_{(i)}$ to $[\N_i]$, which is easily seen to be an isomorphism by noting that the graded dimension of both rings is the same. $\mathbf{\Lambda}$ is also known to be a Hopf algebra, whose primitive elements are exactly the power sum symmetry functions. 
We have thus proved

\begin{theorem}
The Hall algebra $\mathbf{H}_{\Q_0}$ is isomorphic as a Hopf algebra to $\mathbf{\Lambda}$, the ring of symmetric functions. 
\end{theorem}

\section{The quiver of type $A$} \label{typeA}

Let $\Q$ be the quiver of type $A_n$, with arbitrary orientation, such as the case of $A_4$ show in the figure: 
 \setlength{\unitlength}{1pt}
 \begin{center}
 \begin{picture}(130,40)
 \put(0,20){\circle*{5}}
 \put(40,20){\circle*{5}}
 \put(80,20){\circle*{5}}
 \put(120,20){\circle*{5}}
\put(3,20){\vector(1,0){34}}
 \put(43,20){\vector(1,0){34}}
\put(117,20){\vector(-1,0){34}}
  \end{picture}
 \end{center}
We recall that if $\kk$ is a field, the indecomposable representations of $\on{Rep}(\Q,\kk)$ have dimension vectors in which every entry is either $1$ or $0$, and where moreover, we can take all non-zero edge maps to be the identity. If $\V$ is such a representation, then
\[
\V = \V' \otimes_{\Fun} \kk
\]
where $\V' \in \RepQ$ is an indecomposable representation.  Thus  for type A quivers, the functor $()\otimes_{\Fun} \kk $ is a bijection on indecomposable representations. 

\begin{theorem} \label{typeA}
When $\Q$ is a quiver of type $A$, the homomorphism $\rho: \U(\b') \rightarrow \HQe$ defined in section \ref{rho} is an isomorphism.
\end{theorem}
\begin{proof}
By the comment following the proof of theorem \ref{rho_theorem}, it suffices to prove that the corresponding Lie algebra homomorphism
\[
\rho_{Lie} : \b' \rightarrow \b_Q
\]
is an isomorphism. It is clear that $\h$ maps isomorphically to $\h_Q$, so we need to check that the restriction
\[
\rho'_{Lie} : \n_+ \rightarrow \n_Q
\]
is an isomorphism. Since $\n_+$ and $\n_{\Q}$ have the same dimension, it suffices to show that $\rho'_{Lie}$ is surjective, i.e that all $[M]$ with $M$ indecomposable lie in the image of $\rho'_{Lie}$. We proceed by induction on the dimension of $M$. The simples $S_j$ are in the image, and suppose that $\V$ is an indecomposable in $\RepQ$ such that if $\W \in \RepQ$ is indecomposable and $\dim(\W) < \dim(\V)$, then $\W \in \on{Im}(\rho'_{Lie})$. 
The support of $\V$ is connected, and so must be a consecutive chain of vertices. Let $i$ be one of the two extreme vertices. There is then exactly one non-zero edge-map $f_e$ going to/from $i$. Let $\V' \in \RepQ$ be the representation obtained by removing the one-dimensional $V_i$ and setting $f_e=0$. Then $\V'$ is also indecomposable and therefore $[\V'] \in \on{Im}(\rho'_{Lie})$, and we have 
\[
[[S_i], [\V']] = \pm \V
\]    
(depending on whether $e$ is directed towards or away from $i$). Thus $[\V] \in \on{Im}(\rho'_{Lie}) $ as desired. 
\end{proof} 
 
 \begin{remark}
 It follows from theorem \ref{typeA} that the $[M]$ form an integral basis for $\U(\n_+)$. By construction, they have the property that the product of two elements of this basis $[M]\cdot[N]$ is a linear combination of other basis elements with positive integer coefficients. 
 \end{remark}
 
\section{The cyclic quiver} \label{cyclic}

In this section we consider the example of the equi-oriented cyclic quiver of type $A^{(1)}_{n}$, of which the case $n=4$ is pictured below :

 \setlength{\unitlength}{1pt}
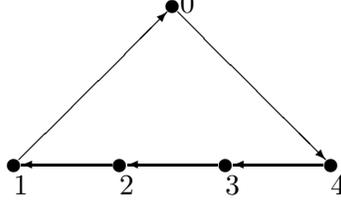
 \begin{center}
 \begin{picture}(130,80)
 \put(0,20){\circle*{5}}
 \put(40,20){\circle*{5}}
 \put(80,20){\circle*{5}}
 \put(120,20){\circle*{5}}
 \put(60,80){\circle*{5}}
\put(37,20){\vector(-1,0){34}}
 \put(77,20){\vector(-1,0){34}}
\put(117,20){\vector(-1,0){34}}
\put(62,78){\vector(1,-1){56}}
\put(2,22){\vector(1,1){56}}
\put(63,78){0}
\put(120,9){4}
\put(80,9){3}
\put(40,9){2}
\put(0,9){1}
  \end{picture}
 \end{center}

As in the type $A$ case above, the situation is essentially identical to that over a field. 
We identify the vertices with residues in $\znz$, and denote the edges by pairs $i,i-1$, so that $f_{i,i-1}: V_i \rightarrow V_{i-1}$. Starting with a nilpotent representation $\V \in \RepQn$ we can form $\mathbf{V} := \oplus_{i \in \znz} V_i$. The edge maps $f_e, e \in \E$ assemble to give a nilpotent endomorphism $x \in \on{End}(\mathbf{V})$, which is equivalent to $\oplus \N_m$. It follows from this that the indecomposable representations in $\RepQn$ are of the form $\Ii_{[k,r]}$, where
\[
(\Ii_{[k,r]})_j = \bigoplus_{\substack{ 0 \leq s \leq r-1 \\ k-s \equiv j}} \{ e_{k-s}, 0 \} 
\]
and
\[
f_{j,j-1} (e_{p}) = e_{p-1} \textrm{ if } p \equiv j
\]
$\Ii_{[k,r]}$ can therefore be visualized as the $r$--dimensional representation which starts at $k$ and winds around the quiver. As over a field, the objects of $\RepQn$ can therefore be identified with $n$--tuples of partitions by the correspondence
\[
(\underline{\lambda_1}, \ldots, \underline{\lambda_n}) \rightarrow \bigoplus_{i \in \znz} \bigoplus_{p} \Ii_{[i,\lambda^j_p]}
\]
The following identity holds in $\HQ$:
\begin{equation}
[\Ii_{[i,p]}] \cdot [\Ii_{[j,q]}] = \left \{ \begin{array}{ll} \left[\Ii_{[j, p+q]}\right] + \left[ \Ii_{[i,p]} \oplus \Ii_{[j,q]} \right] & \textrm{ if } j+q \equiv i \on{mod} n \\
 \left[ \Ii_{[i,p]} \oplus \Ii_{[j,q]} \right] & \textrm{ otherwise } \end{array} \right.
\end{equation} 
from which it follows that in $\n_{\Q}$ we have:
\begin{equation} \label{comm_rel}
\left[ [\Ii_{[i,p]}] , [\Ii_{[j,q]}] \right] = \left \{ \begin{array}{ll}  [ \Ii_{[j, p+q]} ] & \textrm{ if } j+q \equiv i \on{mod} n \\ 0 & \textrm{ otherwise }  \end{array} \right. 
\end{equation}

\bigskip
Let us recall (see eg. \cite{Kac}) that given a complex Lie algebra $\g$ with an invariant symmetric form $\langle, \rangle$, we may form the corresponding affine algebra $$\hat{\g} = \g[t,t^{-1}] \oplus \mathbb{C} c$$ with bracket $$[x \otimes t^n, y \otimes t^m] = [x,y] \otimes t^{n+m} + \langle x, y \rangle n \delta_{n,-m} c.$$ The Lie algebra $\mathfrak{gl}_n$ of $n \times n$ matrices carries the invariant form $\langle x, y \rangle = \on{tr}(xy)$, and so we may form $\hat{\mathfrak{gl}}_n$. $\hat{\mathfrak{gl}}_n$ has a triangular decomposition $$ \agln = \a_- \oplus \h_n \oplus \a_+, $$ where
$\a_+ = t \mathfrak{gl}_n [t] \oplus N_+$, $\a_-= t^{-1} \mathfrak{gl}_n[t^{-1}] \oplus N_-$, $\h_n = D \oplus \mathbb{C}c$, and $N_+, N_-, D$ denote the upper-triangular, lower-triangular, and diagonal matrices in $\mathfrak{gl}_n$ respectively. 

Since $\HQ \simeq \U(\n_{\Q})$, there exists and isomorphism $\HQ \rightarrow \HQ^{op}$ induced by the automorphism $x \rightarrow -x$ at the level of Lie algebras. We denote by $\n^{op}_{\Q}$ the isomorphic Lie algebra whose bracket is the opposite of \ref{comm_rel}. Let 
\begin{align} \label{psi_hom}
\nonumber \psi: & \a_+ \rightarrow \n^{op}_{\Q} \\
\psi(E_{r,s} \otimes t^m  ) & = \left \{ \begin{array}{ll} \Ii_{[r,s-r + mn]} & \textrm{ if } s \geq r \\ \Ii_{[r,n+s-r + mn]} & \textrm{ if } s < r \end{array} \right. 
\end{align}
where $r$ on the right is understood as a residue in $\znz$ (i.e $r=n$ gets mapped to the residue $0$). 

\begin{theorem}
$\psi$ is a Lie algebra isomorphism
\end{theorem}

\begin{proof}
We first check that $\psi$ is a Lie algebra homomorphism by direct calculation. For instance, if $r \leq s$ and $r' \leq s'$, then we have
\begin{align*}
\left[ \psi(E_{r,s} \otimes t^m), \psi(E_{r',s'} \otimes t^{m'})  \right] &= \delta_{s,r'} \psi( E_{r,s'} \otimes t^{m+m'} ) + \delta_{s', r} \psi(  E_{r',s} \otimes t^{m+m'} )  \\
& = \delta_{s,r'} [ \Ii_{[r,s'-r + (m+m')n]}] + \delta_{s',r} [\Ii_{[r',s-r'+ (m+m')n]}] \\
& = \left[ [\Ii_{r,s-r + mn}], \Ii_{[r',s'-r' + m'n]} \right]
\end{align*}
The cases $r \leq s, r' > s'$ and $r > s, r' > s' $ are treated similarly. Now, since the $E_{r,s} \otimes t^m$ (resp. $[\Ii_{k,l}]$) form bases for $\a_+$  (resp. $\n_{\Q}$),  it follows from the definition \ref{psi_hom}  that it takes $\a_+$ isomorphically onto $\n_{\Q}$. 

\end{proof}

We therefore obtain that $\HQ \simeq \HQ^{op} \simeq \U(\a_+)$. Denoting by $\n_+$ the nilpotent half of $\hat{\mathfrak{sl}}_{n-1}$, and using the fact that the positive simple root spaces generate $\n_+$ inside of $\a_+$ as a Lie algebra, we conclude that the Hopf algebra homomorphism $\rho' : \U(\n_+) \rightarrow \HQ$ is injective, and therefore so is $\rho$. To summarize:

\begin{corollary}
$\HQ \simeq \U(\a_+)$, and $\rho: \U(\b') \rightarrow \HQe$ is injective. 
\end{corollary}

\section{Further directions} \label{further}

We collect in this section some missing pieces, and make some comments regarding future directions. 

\begin{enumerate}
\item In the study of quiver representations over a field $\kk$, an important role is played by the Euler form
\[
\left<M,N\right>_a = \sum^{\infty}_{i=1} (-1)^i \on{Ext}^{i}(M,N) = \on{Hom}(M,N) - \on{Ext}^{1}(M,N),
\]
where the last simplification occurs because the $\on{Rep}(\Q, \kk)$ is hereditary.
In particular, it allows to recover the natural inner product on the root lattice of the Kac-Moody algebra $\g(\overline{Q})$ , which appears as $K_0 (\on{Rep}(\Q, \kk))_{nil}$. $K_0 (\RepQn )$ has the right size even over $\Fun$ (see corollary \ref{K0}), but it is not altogether clear whether the Euler form makes sense. 
We can adapt the Yoneda definition of $\on{Ext}$ to define $\on{Ext}^i$ of representations in $\RepQ$, but it is not immediately clear that $\RepQ$ is hereditary, or that Euler form descends to $K_0$. This seems to require the creation of rudimentary homological tools "over $\Fun$". 
\item One would also like to compare the categories $\RepQ$ for different orientations of the quiver. The definition of reflection functors does not seem to go through naively however, since their definition requires one to make sense of the notion of sum $\sum_{e'' = i} f_{e}$ of maps in $\VFun$ having a fixed source/target. 
\item The category $\RepQn$ can sometimes be completely characterized by the poset of submodules of a given module, in which case it is equivalent to an \emph{incidence category} \cite{Sz1}. This leads to a very simple and elementary combinatorial description of the Hall algebra $\HQ$. This theme and several examples are explored in the companion note \cite{Sz2}. 
\item Recently, much beautiful work has been done on Hall algebras $\mathbf{H}(Coh(X, \mathbb{F}_q))$  of the category of coherent sheaves on an algebraic curve $X$ over $ \mathbb{F}_q$ (see \cite{Kap2, S1, S2, S3} for instance ). Using the Dietmar's notion of $\Fun$--scheme \cite{D}, one can make sense of the Hall algebra of coherent sheaves on $\mathbb{P}^1 / \Fun$. One then obtains a $q=1$ version of Kapranov's result (\cite{Kap2}) regarding this object, namely that $\mathbf{H}(Coh(\mathbb{P}^1, \Fun))$ is isomorphic to a non-standard Borel subalgebra in $\hat{\mathfrak{sl}}_2$. The details will appear in \cite{Sz3}. 
\item The category $\RepQ$ shares many structural similarities with incidence categories \cite{Sz1} as well as the category of Feynman graphs \cite{KS}. It may be useful to develop the systematics of "$\Fun$--linear" categories. 
\end{enumerate}

\end{document}